\documentclass[fontsize=11pt]{amsart}

\usepackage{amsmath}
\usepackage{amsthm}
\usepackage{amssymb}
\usepackage{mathtools}
\usepackage{mathrsfs}
\usepackage{physics}
\usepackage[margin=1in]{geometry}
\usepackage{scrextend}
\usepackage{tikz-cd}
\usepackage[shortlabels]{enumitem}
\usepackage{colonequals}
\usepackage{hyperref}
\usepackage{comment}
\usepackage{setspace}
\usepackage{cleveref}

\usepackage[notcolon, notquote]{hanging}

\usepackage[title]{appendix}

\usepackage[notcolon, notquote]{hanging}

\newtheorem{theorem}{Theorem}
\numberwithin{theorem}{section}

\newtheorem{lemma}[theorem]{Lemma}
\newtheorem{proposition}[theorem]{Proposition}
\newtheorem{conjecture}[theorem]{Conjecture}

\theoremstyle{definition}
\newtheorem{remark}[theorem]{Remark}

\theoremstyle{plain}
\newtheorem{Ltheorem}{Theorem}

\newcommand{\F}{\mathbb{F}} 
 
\newcommand{\C}{\mathbb{C}}

\newcommand{\Z}{\mathbb{Z}}

\newcommand{\G}{\mathbb{G}}
\newcommand{\R}{\mathbb{R}}
\renewcommand{\H}{\mathcal{H}}

\renewcommand{\O}{\mathcal{O}}
\renewcommand{\P}{\mathbb{P}}

\newcommand{\Sym}{\text{Sym}}

\newcommand{\I}{\mathcal{I}}
\usepackage{pdflscape}

\DeclareMathOperator{\Span}{Span}

\DeclareMathOperator{\NS}{NS}

\DeclareMathOperator{\Lin}{Lin}

\DeclareMathOperator{\Nef}{Nef}
\DeclareMathOperator{\PGL}{PGL}

\DeclareMathOperator{\Aut}{Aut}
\DeclareMathOperator{\Bl}{Bl}

\begin{document}

\title{Subvarieties of low degree on general hypersurfaces}

\author{Nathan Chen and David Yang}

\begin{abstract}
The purpose of this note is to show that the subvarieties of small degree inside a general hypersurface of large degree come from intersecting with linear spaces or other varieties. This partially generalizes a theorem of Wu.
\end{abstract}

\maketitle

\allowdisplaybreaks

Let $X \subset \P^{n+1}$ be a general hypersurface of degree $d$. There has been considerable interest over the years in understanding what curves or other subvarieties can be found on $X$. It is elementary that $X$ contains lines if and only if $d < 2n$, and Wu \cite{Wu90} observed that if $n = 3$ and $d \ge 6$, then the only curves $C \subseteq X$ of degree $\delta \leq 2d - 2$ are plane sections of $X$. Our first result partially generalizes Wu’s statement to arbitrary dimensions:

\begin{Ltheorem}\label{thm:main}
Let $X \subset \P^{n+1}_{\C}$ be a general hypersurface of degree $d \geq 2n$ and let $Y \subset X$ be a positive-dimensional subvariety with $\deg Y \leq d+2$. Then $\deg Y = d$ and $Y = X \cap \Lambda$ is the intersection of $X$ with a linear subspace $\Lambda \cong \P^{\dim Y + 1}$.
\end{Ltheorem}

\noindent When $d \gg n$, the statement that $\deg Y \geq d$ follows from \cite[Theorem A]{CCZ24}, where the authors establish a more general degree bound for subvarieties of complete intersection varieties. Theorem~\ref{thm:main} answers \cite[Question 6.7]{CCZ24} for hypersurfaces. 

Under stronger degree hypotheses, we prove an analogous statement for subvarieties of higher degree:

\begin{Ltheorem}\label{thm:B}
    Fix any integer $s$. There exists a positive integer $d_{0} = d_{0}(s, n)$ with the following property.\footnote{See Theorem~\ref{thm:curvesHigherDegree} and Remark~\ref{rem:constantd_0} for more precise numerics.} Let $X \subset \P^{n+1}$ be a general hypersurface of degree $d \geq d_{0}$ and let $Y \subset X$ be any subvariety of dimension $m \geq 1$ and degree $\delta \leq ds$. Then $\delta$ is a multiple of $d$ and $Y$ is equal to the generically transverse intersection $Y = X \cap V$ for a unique variety $V \subset \P^{n+1}$ of dimension $m+1$.
\end{Ltheorem}

\noindent In other words, subvarieties of relatively small degree with respect to $d$ all arise as complete intersections of $X$ with proper subvarieties of the ambient projective space, up to embedded points. By \textit{generically transverse} intersection, we mean that $X \cap V$ is generically reduced and $Y$ is equal to $(X \cap V)^{\text{red}}$ as schemes. Note that the intersection may have embedded points if $V$ is singular. For example, let $V \subset \P^{4}$ be the projection of the Veronese surface $\sigma \colon \P^{2} \hookrightarrow \P^{5}$ from a general point lying on the secant variety to $\sigma(\P^{2})$. Then the degree of $V$ is equal to 4 and $V$ is not Cohen-Macaulay at a finite number of double points, so the intersection $V \cap X$ with any smooth hypersurface $X$ passing through one of the singular points of $V$ will be nonreduced.

In a slightly different direction, a theorem of Voisin states that a very general hypersurface $X \subset \P^{n+1}$ of degree $d \geq 2n$ does not contain any rational curves. This story goes back to the work of Clemens \cite{Clemens86} and Ein \cite{Ein88}, and since then there has been a great deal of work by many others on studying the geometric genus of subvarieties in very general hypersurfaces (cf. \cite{Xu94, Pacienza04, CR04, CR19}). For instance, if $X \subset \P^{n+1}$ is a very general hypersurface of degree $d$, then under suitable numerical hypotheses it is known that any subvariety of $X$ must be of general type. Theorem~\ref{thm:main} can be combined with the methods in \cite{Voisin96} to give a bound on the geometric genus of any curve in $X$. There are also some related conjectures of Griffiths--Harris \cite{GH85} from the mid-80s about curves on a very general hypersurface $X \subset \P^{4}$ of degree $d \geq 6$; the strongest form of their conjecture simply asks whether every curve in $X$ is the complete intersection with some surface. This was answered in the negative in a paper of Voisin \cite{Voisin89}, where she constructed counterexamples on any smooth threefold hypersurface.

The starting point for this paper is the observation that the numerics for the geometric genus of curves in very general hypersurfaces closely mirror those coming from Castelnuovo-type bounds (and its variations due to Halphen, Harris, and Gruson--Peskine) for a non-degenerate curve in projective space. In \S1, we will reduce the proof of Theorem~\ref{thm:main} to a special case (Theorem~\ref{thm:planeCurvesMinimal}). Using incidence correspondences and regularity results of Gruson--Lazarsfeld--Peskine \cite{GLP83} for curves, we will prove in \S2 that low degree curves in general hypersurfaces satisfy fairly strong arithmetic genus lower bounds. In \S3, we will show that these lower bounds violate the Castelnuovo-type bounds unless the dimension of the span of the curve is small to begin with, and use this to ultimately prove Theorem~\ref{thm:planeCurvesMinimal}. Finally, the proof of Theorem~\ref{thm:B} will be given in \S4, and involves an induction argument with Theorem~\ref{thm:curvesHigherDegree} as the base case. 

\noindent\textbf{Conventions.} Throughout, we work over the complex numbers. By \textit{variety}, we always mean a projective integral scheme of finite type over $\C$. A \textit{curve} is a variety of dimension 1. We will write $\P^{N_{d}}$ to denote the projective space parametrizing hypersurfaces $X \subset \P^{n+1}$ of degree $d$, where $N_{d} = \binom{d+n+1}{n+1} - 1$. By $\Lin(\P^{r}, \P^{n+1})$, we mean the parameter space of linear embeddings $\P^{r} \hookrightarrow \P^{n+1}$.

Since there are finitely many Hilbert schemes parametrizing curves in $\P^{n+1}$ of bounded degree, the locus of pairs $(X, C)$, where $X$ is a smooth hypersurface and $C$ is a curve of bounded degree, is a finite union of quasi-projective algebraic varieties. Its image in $\P^{N_d}$ under the first projection is a constructible set, and hence contained in a Zariski closed subset of the same dimension. We will often take advantage of this in the following way: if a certain property is shown to be true for the \textit{very general} hypersurface in $\P^{N_{d}}$ (meaning outside of a countable union of subvarieties in $\P^{N_{d}}$), then the previous paragraph can be used to show that it is enough to exclude finitely many \textit{proper} closed subsets of $\P^{N_d}$ from consideration. As a consequence, the same property will then hold for the \textit{general} hypersurface. Thus, in all of the proofs the reader may assume that $X \subset \P^{n+1}$ is a \textit{very} general hypersurface, which is a priori a stronger assumption.

\noindent\textbf{Acknowledgements.} We would like to thank Ben Church, Joe Harris, Rob Lazarsfeld, John Christian Ottem, and Eric Riedl for the many helpful conversations. We would also like to thank the referees for their valuable suggestions and for pointing out a gap in an argument from an earlier draft. During the preparation of this article, N.C. was partially supported by NSF grant DMS-2103099 and D.Y. was partially supported by NSF grant DMS-2402082.

\section{Genus bounds for curves on very general hypersurfaces}

In order to prove the results stated in the introduction, in the next three sections we will aim to prove a special case:

\begin{theorem}\label{thm:planeCurvesMinimal}
Let $X \subset \P^{n+1}_{\C}$ be a general hypersurface of degree $d \geq 2n$ and let $C \subset X$ be a curve of degree $\delta$. If $\delta \leq d+2$, then $\delta = d$ and $C = X \cap \Lambda$ for some 2-plane $\Lambda$.
\end{theorem}

Theorem~\ref{thm:main} follows immediately from this:

\begin{proof}[Proof of Theorem~\ref{thm:main}]

As in the statement of the theorem, let $X \subset \P^{n+1}$ be a general hypersurface of degree $d \geq 2n$ and let $Y \subset X$ be a positive-dimensional subvariety such that $\deg Y \leq d+2$. Consider general hyperplanes $H_{1}, \ldots, H_{s}$ in $\P^{n+1}$, where $s = \dim Y - 1$. Slicing down $Y$, we then obtain a curve
\[ V \coloneqq Y \cap H_{1} \cap \cdots \cap H_{s} \]
of degree at most $d+2$ which is contained in $X$ (by Bertini's theorem $V$ is integral, but we will not need this). Theorem~\ref{thm:planeCurvesMinimal} implies that $\deg V = d$ and $V$ is a plane curve, i.e. $\dim \Span V = 2$. At each step of slicing, the degree stays the same and the span decreases by 1 since the $H_{i}$ are general. Thus, working backwards we see that $\deg Y = d$ and $\dim \Span(Y) = 2 + s = \dim Y + 1$. Setting $\Lambda \coloneqq \Span(Y)\cong \P^{\dim Y + 1}$, it follows that $Y = X \cap \Lambda$.
\end{proof}

An invariant that will feature prominently in the proof of Theorem~\ref{thm:planeCurvesMinimal} is the dimension of the span of $C$ inside $\P^{n+1}_{\C}$. As mentioned in the introduction, a general hypersurface $X \subset \P^{n+1}$ of degree $d \geq 2n$ does not contain any lines, so $\dim \Span(C) \geq 2$. In fact, we claim that it suffices to show $\dim \Span(C) = 2$:

\begin{proposition}\label{prop:incidence2plane}
On a general hypersurface $X \subset \P^{n+1}$ of degree $d \geq 2n$, the intersection $X \cap \Lambda$ with every 2-plane $\Lambda$ is an integral curve.
\end{proposition}

\begin{proof}
By a dimension count, we see that $X$ does not contain any 2-planes. When $n = 2$, the Noether--Lefschetz theorem implies that a very general surface $X$ satisfies the desired property. By the last paragraph in the Conventions section, this extends to the general surface. Assuming $n \geq 3$, we will estimate the space of hypersurfaces for which the intersection with a 2-plane $\Lambda$ splits off a plane curve of degree $a < d$. Consider the incidence correspondence:
\begin{center}
\begin{tikzcd}
\Phi_{a} \arrow[r, phantom, "\coloneqq"] &[-0.2in] {\left\{ (X_{d} , f \colon \P^{2} \rightarrow \P^{n+1}, C) \in \P^{N_{d}} \times \Lin(\P^{2}, \P^{n+1}) \times |\O_{\P^{2}}(a)| \colon f(C) \subset X_{d} \right\}}
\end{tikzcd}
\end{center}
There is an action of $\PGL_{3} = \Aut(\P^{2})$ on $\Phi_{a}$ which is defined by $\phi \cdot (X, f, C) = (X, f \circ \phi, \phi^{-1}(C))$. Since we are interested in the orbit space, we will pass to the quotient $\Phi_{a}/\PGL_{3}$, which then admits the following projection maps:
\begin{center}
\begin{tikzcd}[column sep=small]
    & \Phi_{a}/\PGL_{3} \arrow[ld, swap, "\pi_{1}"] \arrow[rd, "\pi_{2}"] & \\
    \P^{N_{d}} & & {\left(\Lin(\P^{2}, \P^{n+1}) \times |\O_{\P^{2}}(a)| \right)/\PGL_{3}}
\end{tikzcd}
\end{center}
If $\dim \Phi_{a}/\PGL_{3} < N_{d}$, then $\Phi_{a}/\PGL_{3}$ cannot dominate $\P^{N_{d}}$ and so we are done.

When $n \geq 3$, a very general $X$ does not contain rational curves \cite{Voisin96}, so we may assume $a \geq 3$. Fixing a point in the image of $\pi_{2}$ amounts to fixing the 2-plane $\Lambda$ and a plane curve $C \subset \Lambda$ of degree $a$. The fiber over such a point is isomorphic to the projectivization of $H^{0}(\P^{n+1}, \I_{C}(d))$. The latter group is the kernel of the restriction map $H^{0}(\P^{n+1}, \O_{\P^{n+1}}(d)) \twoheadrightarrow H^{0}(\Lambda, \O_{\Lambda}(d)) \twoheadrightarrow H^{0}(C, \O_{C}(d))$, so $\pi_{2}$ has relative dimension equal to $h^{0}(\P^{n+1}, \O_{\P^{n+1}}(d)) - h^{0}(C, \O_{C}(d)) - 1$. On the other hand, the image of $\pi_{2}$ has dimension at most
\begin{align*}
\dim \left[ \left(\Lin(\P^{2}, \P^{n+1} \right) \times |\O_{\P^{2}}(a)|)/\PGL_{3} \right] = 3(n-1) + h^{0}(\O_{\P^{2}}(a)) - 1.
\end{align*}
Therefore,
\begin{align*}
\dim \Phi_{a}/\PGL_{3}&\leq 3(n-1) + h^{0}(\O_{\P^{2}}(a)) - 1 + h^{0}(\P^{n+1}, \O_{\P^{n+1}}(d)) - h^{0}(C, \O_{C}(d)) - 1 \\
&=N_d+3n-3-a(d-a).
\end{align*}
From this we see that $\dim \Phi_{a}/\PGL_{3} < N_{d}$ if $a(d-a)>3n-3$, which is true when $\frac{d}{2} \geq a \geq 3$ (by symmetry it is enough to check this range) and $d \geq 2n$.
\end{proof}

\noindent By further slicing with general hyperplanes, it follows that the intersection of a general hypersurface $X \subset \P^{n+1}$ of degree $d \geq 2n$ with \textit{any} linear subspace of dimension at least $2$ is reduced and irreducible.

The argument above can be generalized to the case of quadric surfaces:

\begin{proposition}\label{prop:incidenceQuadricSurf}
On a general hypersurface $X \cap \P^{n+1}$ of degree $d \geq 2n$, the intersection $X \cap Q$ with every irreducible quadric surface $Q$ is an integral curve.
\end{proposition}

\begin{proof}
By a dimension count, $X$ does not contain any quadric surfaces. As mentioned earlier, the case $n = 2$ follows from the Noether--Lefschetz theorem. Assuming $n \geq 3$, we will estimate the dimension of the space of hypersurfaces for which the intersection with a quadric surface splits off a component.
    
Let us first work out the case where $Q \cong \P^{1} \times \P^{1}$ is a smooth quadric surface. For two nonnegative integers $a$ and $b$, consider the incidence correspondence:
\begin{center}
\begin{tikzcd}
\Phi_{a,b} \arrow[r, phantom, "\coloneqq"] &[-0.2in] {\left\{ (X_{d} , f \colon \P^{1} \times \P^{1} \rightarrow \P^{n+1}, C) \in \P^{N_{d}} \times \Lin(\P^{1}\times \P^{1}, \P^{n+1}) \times |\O_{\P^{1}\times\P^{1}}(a,b)| \colon f(C) \subset X_{d} \right\}}.
\end{tikzcd}
\end{center}
Here, an element of $\Lin(\P^{1}\times \P^{1}, \P^{n+1})$ consists of the Segre embedding $\P^{1} \times \P^{1} \hookrightarrow \P^{3}$ composed with a linear embedding $\P^{3} \hookrightarrow \P^{n+1}$. There is an action of $\Aut(\P^{1} \times \P^{1})$ on $\Phi_{a, b}$ which is defined by $\phi \cdot (X, f, C) = (X, f \circ \phi, \phi^{-1}(C))$. Since we are interested in the orbit space, we will pass to the quotient $\Phi_{a,b}/\Aut(\P^{1} \times \P^{1})$, which then admits the following projection maps
\begin{center}
\begin{tikzcd}[column sep=small]
    & \Phi_{a,b}/\Aut(\P^{1} \times \P^{1}) \arrow[ld, swap, "\pi_{1}"] \arrow[rd, "\pi_{2}"] & \\
    \P^{N_{d}} & & {\left(\Lin(\P^{1}\times\P^{1}, \P^{n+1}) \times |\O_{\P^{1}\times\P^{1}}(a,b)| \right)/\Aut(\P^{1} \times \P^{1})}
\end{tikzcd}
\end{center}

We wish to show that if $(a,b)\neq(d,d)$, then $\Phi_{a,b}/\Aut(\P^{1} \times \P^{1})$ cannot dominate $\P^{N_{d}}$. When $n \geq 3$, a very general $X$ does not contain rational curves \cite{Voisin96}, so we may assume by symmetry that $d-2\geq b \geq a \geq 2$. It suffices to show that in this case, $\dim \Phi_{a,b}/\Aut(\P^{1} \times \P^{1}) < N_{d}$.

 Fixing a point in the image of $\pi_{2}$ amounts to fixing a quadric $Q$ and a curve $C \subset \P^{1} \times \P^{1}$ of bi-degree $(a,b)$. The fiber over such a point is isomorphic to the projectivization of $H^{0}(\P^{n+1}, \I_{C}(d))$. The latter group is the kernel of the restriction map $H^{0}(\P^{n+1}, \O_{\P^{n+1}}(d)) \twoheadrightarrow H^{0}(Q, \O_{Q}(d)) \twoheadrightarrow H^{0}(C, \O_{C}(d))$, so $\pi_{2}$ has relative dimension equal to $h^{0}(\P^{n+1}, \O(d)) - h^{0}(C, \O(d)) - 1$. On the other hand, the image of $\pi_{2}$ has dimension at most
\begin{align*}
\dim [(\Lin(\P^{1} \times \P^{1}, \P^{n+1}) \times |\O_{\P^{1}\times\P^{1}}(a,b)|)/\Aut(\P^{1} \times \P^{1})] = 4(n-2) + 9 + h^{0}(\O_{\P^{1}\times\P^{1}}(a,b)) - 1.
\end{align*}
It follows that we have
\begin{align*}
\dim \Phi_{a,b}/\Aut(\P^{1}\times\P^{1}) &\leq 4(n-2) + 9 + h^{0}(\O_{\P^{1}\times\P^{1}}(a,b)) - 1 + h^{0}(\P^{n+1}, \O(d)) - h^{0}(C, \O(d)) - 1 \\
&= N_d + 4n + h^{0}(\O_{\P^{1}\times\P^{1}}(a,b)) - h^{0}(C, \O_{C}(d))
\end{align*}
Since $h^{0}(C, \O_{C}(d)) = h^{0}(\P^{1} \times \P^{1}, \O_{\P^{1} \times \P^{1}}(d,d)) - h^{0}(\P^{1} \times \P^{1}, \O_{\P^{1} \times \P^{1}}(d-a,d-b))$, we see that $\dim \Phi_{a,b}/\Aut(\P^{1}\times\P^{1}) < N_{d}$ if
\[ h^{0}(\O_{\P^{1}\times\P^{1}}(d,d)) - h^{0}(\O_{\P^{1}\times\P^{1}}(d-a,d-b)) - h^{0}(\O_{\P^{1}\times\P^{1}}(a,b)) -4n > 0. \]
This simplifies to checking that $(a+b)d - 2ab - 1 - 4n > 0$, which is true when $2 \leq a \leq b \leq d-2$, $d \geq 2n$, and $n \geq 3$.

Let us now consider the case where $Q$ is a singular quadric cone (spanning $\P^{3}$). The blow-up of $Q$ at the singular point is a Hirzebruch surface $\pi \colon \F_{2} \rightarrow \P^{1}$ which satisfies $\NS(\F_{2}) = \Z \langle F, E \rangle$, where $F$ is the class of a fiber of $\pi$ and $E$ is the unique section with negative self-intersection. These satisfy the relations $F^{2} = 0$, $F \cdot E = 1$, and $E^{2} = -2$. The morphism $\F_{2} \rightarrow Q \subset \P^{3}$ is then given by the linear series $|E + 2F|$. We will now recall some facts:
\begin{enumerate}[wide, labelwidth=0pt, labelindent=0pt, label=(\textbf{\roman*})]
\item The nef cone of $\F_{2}$ is given by $\Nef(\F_{2}) = \R_{\geq 0} \cdot F + \R_{\geq 0} \cdot (E + 2F)$.
\item If $C \subset \F_{2}$ is an integral curve, then either $C = E$ or $C \in \Nef(\F_{2})$.
\item If $aE + bF$ is a line bundle on $\F_{2}$ with $b \geq 2a\geq 0$, then $E + 2F = \O_{\F_{2}}(1)$ under the identification $\F_{2} \cong \P(\O_{\P^{1}} \oplus \O_{\P^{1}}(2))$ so we can compute global sections of $aE + bF$ by pushing forward to $\P^{1}$ as follows.
\begin{align*}
h^{0}(\F_{2}, aE + bF) &= h^{0}(\F_{2}, a(E + 2F) + (b-2a)F) \\
&= h^{0}(\P^{1}, \Sym^{a}(\O_{\P^{1}} \oplus \O_{\P^{1}}(2)) \otimes \O_{\P^{1}}(b-2a)) \\
&= \sum_{k=0}^{a} h^{0}(\P^{1}, \O_{\P^{1}}(b-2a+2k)) \\
&= (a+1)(b-a + 1), \qquad \text{when } b \geq 2a.
\end{align*}
On the other hand, if $0 \leq b \leq 2a$ and $a\geq 0$, then this sum actually only depends on $b$:
\begin{align*}
h^{0}(\F_{2}, aE + bF) &= \sum_{k=\lceil \frac{2a-b}{2} \rceil}^{a} h^{0}(\P^{1}, \O_{\P^{1}}(b-2a+2k)) \\
&= \begin{cases} 1 + 3 + \cdots + (b+1) & \text{if $b$ is even,} \\
2 + 4 + \cdots + (b+1) & \text{if $b$ is odd;} \end{cases} \\
&= \begin{cases} \frac{(b+2)^{2}}{4} & \text{if $b$ is even,} \\
\frac{(b+1)(b+3)}{4} & \text{if $b$ is odd}. \end{cases}
\end{align*}
This makes sense in light of the fact that $h^{0}(\F_{2}, (a-1)E + bF) = h^{0}(\F_{2}, aE + bF)$ when $b < 2a$.
\end{enumerate}

Now following the same strategy as before, consider the incidence:
\[ \Psi_{a,b} \coloneqq \{ (X_{d}, f \colon \F_{2} \rightarrow \P^{n+1}, C) \in \P^{N_{d}} \times \mathcal{V} \times |\O_{\F_{2}}(aE+bF)| \colon f(C) \subset X_{d} \}. \]
Here, $\mathcal{V}$ is the space of maps that factor through a singular quadric $Q \subset \P^{n+1}$, and note that the space of singular quadrics in $\P^{3}$ has dimension 8. As before, it suffices to show that $\dim \Psi_{a,b}/\Aut(\F_{2}) < N_{d}$. A nearly identical computation as the one above reduces to checking that
\begin{equation}\label{eq:propQuadricCheck}
h^{0}(\O_{\F_{2}}(dE + 2dF)) - h^{0}(\O_{\F_{2}}(aE + bF)) - h^{0}(\O_{\F_{2}}((d-a)E + (2d-b)F)) - 4n - 1 > 0
\end{equation}
when $2(d-2)\geq b \geq 2a$ and $a \geq 2$. Our assumptions on $a$ and $b$ follow from the description of the effective classes in (ii) together with the fact that we are only interested in breaking off integral non-rational curves $C$. Using (iii), we can write
\begin{align*}
    &h^{0}(\O_{\F_{2}}(dE + 2dF)) - h^{0}(\O_{\F_{2}}(aE + bF)) - h^{0}(\O_{\F_{2}}((d-a)E + (2d-b)F))\\
    &= (d+1)^2-(a+1)(b-a+1)-h^{0}(\O_{\F_{2}}((d-a)E + (2d-b)F))\\
    &\geq (d+1)^2-(a+1)(b-a+1)-(2d-b+2)^2/4\\
    &\geq (d+1)^2-(\frac{b+2}{2})^2-(\frac{2d-b+2}{2})^2\\
    &\geq (d+1)^2-3^2-(d-1)^2\\
    &=4d-9,
\end{align*}
which implies \eqref{eq:propQuadricCheck} since $d\geq 2n\geq 6$.
\end{proof}

An ingredient that we will frequently use later is Castelnuovo's bound for the arithmetic genus of a non-degenerate curve in projective space:

\begin{theorem}[Castelnuovo]\label{prop:Castelnuovo'sThm}
Fix $r \geq 3$ and let $C \subset \P^{r}$ be a non-degenerate integral curve of degree $\delta$. If we set $M \coloneqq \lfloor \frac{\delta-1}{r-1} \rfloor$, then
\[ p_{a}(C) \leq M \cdot \left( \delta - \frac{M+1}{2} (r-1) - 1 \right). \]
\end{theorem}

\noindent Setting the right-hand side equal to $f(M)$, it is straightforward to check that the function $f(m)$ is maximized when $m = \frac{\delta-1}{r-1} - \frac{1}{2}$. Thus, for any $r \geq 3$ we have
\begin{equation}\label{eq:anyr>=3Castelnuovo}
p_{a}(C) \leq f\left( \frac{\delta-1}{r-1} - \frac{1}{2} \right) = \frac{1}{2(r-1)} \delta^{2} - \frac{(r+1)}{2(r-1)} \delta + \frac{(r+1)^{2}}{8(r-1)}.
\end{equation}

When $r$ is even, the expression $\frac{\delta-1}{r-1} - \frac{1}{2}$ can never be an integer so there is a slight improvement: $f(m)$ is maximized when $m = \frac{\delta-1}{r-1} - \frac{r}{2(r-1)}$ (or equivalently when $m = \frac{\delta-1}{r-1} - \frac{r-2}{2(r-1)}$). For future reference, we will write out this inequality for $r = 4$:
\begin{equation}\label{eq:r=4Castelnuovo}
p_{a}(C) \leq f\left( \frac{\delta-1}{r-1} - \frac{r}{2(r-1)} \right) \Big|_{r=4} = \frac{1}{6} \delta^{2} - \frac{5}{6} \delta + 1.
\end{equation}
For $r = 3$, one can show that
\begin{equation}\label{eq:r=3Castelnuovo}
p_{a}(C) \leq \frac{1}{4} \delta^{2} - \delta + 1.
\end{equation}

We will also need a number of refinements of Castelnuovo's theorem. For instance, Halphen \cite{Halphen21} proved that a non-degenerate curve $C \subset \P^{3}$ of degree $\delta$ which lies on a cubic surface but not a quadric surface satisfies
\begin{align}\label{eq:HalphenCubicSurf}
p_{a}(C) \leq \begin{cases}
    \frac{1}{6} \delta^{2} - \frac{1}{2} \delta + 1 & \text{ if } \delta \equiv 0 \pmod{3} \\
    \frac{1}{6} \delta^{2} - \frac{1}{2} \delta + \frac{1}{3} & \text{ if } \delta \not\equiv 0 \pmod{3} \\
\end{cases}
\end{align}
More generally, Gruson--Peskine \cite{GP78} showed that if $C \subset \P^{3}$ is a non-degenerate curve of degree $\delta > s(s-1)$ which is not contained in a surface of degree less than $s$, then $p_{a}(C) \leq 1 + \frac{\delta}{2}\left( s + \frac{\delta}{s} - 4 \right)$. (In fact, they prove a slightly stronger bound, but the above weaker inequality is simpler to state and suffice for our purposes). Evaluating this with $s = 3$ and $4$, we see that if $C \subset \P^{3}$ is a non-degenerate curve not contained in a surface of degree at most $3$, then
\begin{align}\label{eq:GrusonPeskineNotInQuadricCubicSurf}
p_{a}(C) \leq \begin{cases}
\lfloor \frac{1}{8} \delta^{2} \rfloor + 1 & \text{ if } \delta > 12; \\ \lfloor \frac{1}{6} \delta^{2} - \frac{1}{2} \delta \rfloor + 1 & \text{ if } 12 \geq \delta > 6.
\end{cases}
\end{align}
Finally, the main theorem of Harris \cite{Harris80} says that if $C$ is contained in an irreducible surface $S \subset \P^{3}$ of degree $k$, then
    \begin{equation}\label{eq:r=3Harris}
    p_{a}(C) \leq \begin{cases} \pi(\delta, k) = \frac{\delta^{2}}{2k} + \frac{1}{2} (k-4) \delta + 1 - \frac{\varepsilon}{2} (k-\varepsilon-1+\frac{\varepsilon}{k}) & \text{if } \delta > k(k-1), \\ \pi(\delta, \left\lfloor \frac{\delta-1}{k} \right\rfloor + 1) & \text{if } \delta \leq k(k-1); \end{cases}
    \end{equation}
    where $0 \leq \varepsilon \leq k-1$ is such that $\delta+\epsilon$ is a multiple of $k$. Note that $-\frac{\varepsilon}{2}(k-\varepsilon-1+\frac{\varepsilon}{k}) \leq 0$, so
    \[ \pi(\delta, k) \leq \frac{\delta^{2}}{2k} + \frac{1}{2} (k-4)\delta + 1 \eqqcolon R(\delta, k). \]

\section{Incidence correspondences}

In this section, we will set up a number of incidence correspondences and use these to deduce various lower bounds for the arithmetic genus of a low degree curve on a general hypersurface $X \subset \P^{n+1}$. A key ingredient in this estimate is a theorem of Gruson--Lazarsfeld--Peskine \cite{GLP83}, which says that hypersurfaces of degree $d$ trace out a complete linear system on the curves we are interested in:

\begin{lemma}\label{lem:IncidenceFiberDim}
    Fix a positive integer $d$ and let $C \subset \P^{n+1}$ be an integral curve of degree $\delta \leq d+2$. Suppose that $\Span(C) \cong \P^{r}$ for some positive integer $r \geq 3$. Then
    \[ h^{0}(\P^{n+1}, \I_{C}(d)) = h^{0}(\P^{n+1}, \O(d)) - (d \delta + 1 - p_{a}(C)). \]
\end{lemma}

\begin{proof}

By assumption, $C \subset \P^{r}$ is non-degenerate and $\deg C = \delta$. Since $d \geq \delta + 1 - r$ (since $r \geq 3$), by the main theorem of \cite{GLP83} we get a surjection $H^{0}(\P^{r}, \O(d)) \twoheadrightarrow H^{0}(C, \O_{C}(d))$. By combining this with the obvious surjection $H^{0}(\P^{n+1}, \O(d)) \twoheadrightarrow H^{0}(\P^{r}, \O(d))$, we have an exact sequence
\begin{equation}\label{eq:exactseqI_C(d)}
0 \rightarrow H^{0}(\P^{n+1}, \I_{C}(d)) \rightarrow H^{0}(\P^{n+1}, \O(d)) \rightarrow H^{0}(C, \O_{C}(d)) \rightarrow 0.
\end{equation}
Thus, we reduce our problem to computing the quantity $h^{0}(C, \O_{C}(d))$. Before doing so, it will be useful to introduce some additional terminology. Recall that a coherent sheaf $\mathcal{F}$ on $\P^{r}$ is said to be \textit{$n$-regular} if $H^{i}(\P^{r}, \mathcal{F}(n-i)) = 0$ for $i > 0$. By Mumford's theorem \cite[Theorem 1.8.3]{Lazarsfeld04a}, this implies that $\mathcal{F}$ is $(n+k)$-regular for all $k \geq 0$. We will say that a curve $C \subset \P^{r}$ is $n$-regular if its ideal sheaf $\I_{C}$ is.

From \cite[Theorem 1.1]{GLP83}, we know that a non-degenerate curve $C \subset \P^{r}$ of degree $\delta$ is $(\delta+2-r)$-regular. Since $r \geq 3$, by our assumption we see that $H^{1}(\P^{r}, \I_{C} (d)) = 0 = H^{2}(\P^{r}, \I_{C} (d))$ for $d \geq \delta - 2$. Let us apply this to the long exact sequence on cohomology associated to the ideal sheaf sequence for $C \subset \P^{r}$:
\[ 0 = H^{1}(\P^{r}, \I_{C}(d)) \rightarrow H^{1}(\P^{r}, \O(d)) \rightarrow H^{1}(C, \O_{C}(d)) \rightarrow H^{2}(\P^{r}, \I_{C}(d)) = 0. \]
This implies that $H^{1}(C, \O_{C}(d)) \cong H^{1}(\P^{r}, \O(d)) = 0$, and Riemann--Roch gives
\begin{equation}\label{eq:cohomologyO_C(d)}
h^{0}(C, \O_{C}(d)) = \chi(C, \O_{C}(d)) = d \delta + 1 - p_{a}(C).
\end{equation}
Putting together \eqref{eq:exactseqI_C(d)} and \eqref{eq:cohomologyO_C(d)} gives the desired result.
\end{proof}

We will use the previous lemma to give a lower bound on the arithmetic genus of a low degree curve on a general hypersurface.

\begin{proposition}\label{prop:ArithmeticGenusIncidence}
Let $X \subset \P^{n+1}$ be a general hypersurface of degree $d$ and let $C \subset X$ be an integral curve of degree $\delta \leq d+2$. Let $f \colon C' \rightarrow C$ be the normalization and suppose that $\dim \Span(C) \leq r$. Then
\begin{equation}\label{eq:ArithmeticGenusFullIncidence}
p_{a}(C) \geq d \delta - (r+1)(n+1-r+h^{0}(f^{\ast}\O_{C}(1))) - 4 p_{g}(C) + 5
\end{equation}
\end{proposition}

\begin{proof}
Let $g'$ be the geometric genus of $C$. Let $\mathcal{M}_{g'}(\mathbb{P}^{n+1},d)$ be the stack of degree $d$ maps from a (smooth) curve of genus $g'$ to $\mathbb{P}^n$. Fix a substack $\mathcal{M}\subseteq\mathcal{M}_{g'}(\mathbb{P}^{n+1},d)$ whose general member is a generically injective map $i:C'\rightarrow\mathbb{P}^{r}$ such that the image curve $C = i(C')$ has arithmetic genus $g$. We begin with the following.

\textbf{Claim.} The dimension of $\mathcal{M}$ is bounded from above by $4g'-4 + (r+1) \cdot h^{0}(i^{\ast}\O(1))$.

To see this, note that the dimension of $\mathcal{M}$ is bounded from above by the dimension of the tangent space to $\mathcal{M}_{g'}(\mathbb{P}^{n+1},d)$ at $[i] \in \mathcal{M}$, which is bounded from above by $h^{0}(i^{\ast}T_{\P^{r}}) + 3g' - 3$. By pulling back the Euler sequence to get $0 \rightarrow \O_{C'} \rightarrow i^{\ast}\O(1)^{\oplus r+1} \rightarrow i^{\ast}T_{\P^{r}} \rightarrow 0$, we see that $h^{0}(i^{\ast}T_{\P^{r}}) \leq (r+1) \cdot h^{0}(f^{\ast}\O_{C}(1)) - h^{0}(\O_{C'}) + h^{1}(\O_{C'}) = (r+1) \cdot h^{0}(f^{\ast}\O_{C}(1)) + g' - 1$. Combining these two inequalities proves the claim.

Next, consider the incidence variety
\[ \Psi \coloneqq \left\{ (X, f \colon \P^{r} \rightarrow \P^{n+1}, i\colon C'\rightarrow\mathbb{P}^r) \in \P^{N_{d}} \times \Lin(\P^{r}, \P^{n+1}) \times \mathcal{M} \mid f(i(C')) \subset X \right\}. \]
There is a natural action of $\PGL_{r+1} = \Aut(\P^{r})$ on $\Psi$ via $\phi\cdot (X, f, i) \mapsto (X, f\circ \phi^{-1}, \phi\circ i)$, which gives $\Psi$ the structure of a $\PGL_{r+1}$-bundle. Since we are interested in the orbit space, let us pass to the quotient, which admits the following maps:
\begin{center}
\begin{tikzcd}[column sep=small]
    & {\Psi / \PGL_{r+1}} \arrow[ld, swap, "\pi_{1}"] \arrow[rd, "\pi_{2}"] & \\
\P^{N_{d}} & & {(\Lin(\P^{r},\P^{n+1}) \times \mathcal{M})/\PGL_{r+1}}
\end{tikzcd}
\end{center}
The fact that $X$ is general implies that $\Psi/\PGL_{r+1}$ must dominate $\P^{N_{d}}$ under the first projection. Note that the dimension of the fiber of $\pi_{2}$ is given by Lemma~\ref{lem:IncidenceFiberDim}. On the other hand, ${(\Lin(\P^{r},\P^{n+1}) \times \mathcal{M})/\PGL_{r+1}}$ admits a map to the Grassmannian $\G(r, n+1)$, whose fibers have dimension bounded from above by the Claim. This implies that
\begin{align*}
N_{d} \leq \dim \Psi/\PGL_{r+1} &\leq N_{d} - (d\delta + 1 - g) + \dim \G(r,n+1) + 4g'-4 + (r+1) \cdot h^{0}(i^{\ast}\O(1)).
\end{align*}
Rearranging the terms and solving for $g$, we arrive at the desired inequality.
\end{proof}

\begin{remark}
In the proposition above, we did not assume that $C \subset \P^{r}$ is non-degenerate. In fact, choosing different values of $r$ may give better lower bounds. Setting $r = n+1$ amounts to studying a simpler incidence involving just pairs $(X, C)$, and leads to the following inequality:
\begin{equation}\label{eq:r=n+1arithmeticgenusbound}
p_{a}(C) \geq d \delta - (n+2) h^{0}(f^{\ast}\O_{C}(1)) - 4 p_{g}(C) + 5.
\end{equation}
\end{remark}

Finally, we conclude this section with some incidence arguments involving curves in general hypersurfaces which are also contained in low degree surfaces. We will need to fix some notation. Let $\H_{P}$ be a locally closed subscheme of some Hilbert scheme $\H$, such that each point of $\H_P$ corresponds to an integral surface $S \subset \P^{r}$ with Hilbert polynomial $P$. Let $\H_{\delta,g}$ be the locally closed subscheme of another Hilbert scheme, which parametrizes integral curves $C \subset \P^{r}$ of degree $\delta$ and arithmetic genus $g$. We will prove the following.

\begin{proposition}\label{prop:arithmeticGenusIncidenceSurf}
    Let $X \subset \P^{n+1}$ be a general hypersurface of degree $d$ and let $C \subset X$ be an integral curve of degree $\delta \leq d+2$. Suppose that $\Span(C) \cong \P^{r}$ for some $r \geq 3$, and furthermore assume that $C$ is contained in a surface $S \subset \P^{r}$ belonging to $\H_{P}$. Then
    \[ p_{a}(C) \geq d \delta - \dim \G(r,n+1) - \dim \H_{P} - \frac{\delta^{2}}{\deg S} \]
\end{proposition}

\begin{proof}
    Consider the incidence variety
    \[ \Phi_{P,\delta,g,r}^{d} \coloneqq \left\{ (X, f\colon \P^{r} \rightarrow \P^{n+1}, S, C) \in \P^{N_{d}} \times \Lin(\P^{r}, \P^{n+1}) \times \H_{P} \times \H_{\delta,g} \middle| f(C) \subset f(S) \cap X \right\}.
    \]
Note that $\PGL_{r+1}$ acts on the last three components in a natural way. Passing to the quotient, we arrive at the following maps:
\begin{center}
\begin{tikzcd}
    & {\Phi_{P,\delta,g,r}^{d}/\PGL_{r+1}} \arrow[ld, swap, "\pi_{1}"] \arrow[rd, "\pi_{2}"] & \\
    \P^{N_{d}} & & {\left(\Lin(\P^{r}, \P^{n+1}) \times \H_{P} \times \H_{\delta,g}\right)/\PGL_{r+1}}
\end{tikzcd}
\end{center}
We know that $\Phi_{P,\delta,g,r}^{d}/\PGL_{r+1}$ must dominate $\P^{N_{d}}$ under the first projection. The dimension of any fiber of $\pi_{2}$ is given by Lemma~\ref{lem:IncidenceFiberDim}. On the other hand, the image of $\pi_{2}$ lies in the locus $Z$ of triples $(f,S,C)$ with $C\subseteq S$. Let us bound the dimension of $Z$.

The variety $(\Lin(\P^{r}, \P^{n+1}) \times \H_{P})/\PGL_{r+1}$ has dimension equal to $\dim \G(r,n+1) + \dim \H_{P}$, so it suffices to bound the dimension of the fiber of $Z$ above any fixed pair $(f,S).$ This fiber can be identified with the closed subscheme $\mathcal{H}_{\delta,g}^S\subseteq\mathcal{H}_{\delta,g}$ parametrizing those $C$ that lie inside $S$. If we take a minimal resolution $\nu \colon S' \rightarrow S$, there will be a dense open subscheme $U\subseteq\mathcal{H}_{\delta,g}^S$ such that the universal family of curves $C_U\subseteq S\times U$ lifts to a universal family of curves $C'_U\subseteq S'\times U$ in a way so that for any point $p\in U$, the fiber $C'_p$ of $C'_U$ over $p$ is the strict transform of the fiber $C_p$ of $C_U$ over $p$. The dimension of $U$ (equivalently, of $\mathcal{H}_{\delta,g}^S$) at $p$ is bounded from above by the dimension of the tangent space to the corresponding Hilbert scheme of $S'$, which is given by $h^{0}(N_{C'_p/S'}) = h^{0}(\O_{C'_p}(C'_p))$. Let us abuse notation and suppress the dependence on $p$, i.e., we will write $C'$ for $C'_p.$

To bound $h^{0}(\O_{C'}(C'))$, we note the following Lemma: If $C'$ is an integral Gorenstein curve and $D$ is an effective divisor on $C'$, then $h^{0}(C', D) \leq \deg D+1$. The line bundle $\O_{C'}(C')$ has degree equal to $(C')^{2}$ which, by the Hodge Index theorem on $S'$, satisfies $(C')^{2} \leq \frac{(C' \cdot \nu^{\ast}\O_{S}(1))^{2}}{\O_{S}(1)^{2}} = \frac{\delta^{2}}{\deg S}$. Plugging this into the lemma, we find that
\[ h^{0}(N_{C'/S'}) \leq \max \left\{ 0, \frac{\delta^{2}}{\deg S} + 1 \right\} \leq 1 + \frac{\delta^{2}}{\deg S}. \]
Combining all of the above ingredients, we see that
\begin{align*}
N_{d} &\leq \dim \Phi_{P,\delta,g,r}^{d}/\PGL_{r+1} \\
&\leq N_{d} - (d\delta + 1 - p_{a}(C)) + \dim \G(r,n+1) + \dim \H_{P} + 1 + \frac{\delta^{2}}{\deg S}
\end{align*}
Rearranging the terms and solving for $p_{a}(C)$ gives the desired bound.
\end{proof}

\section{Proof of Theorem~\ref{thm:planeCurvesMinimal}}

The case where $n = 1$ is trivial, the case where $n = 2$ follows from the Noether--Lefschetz theorem, and the case where $n = 3$ was proved by Wu \cite{Wu90}. Thus, we may assume that $n \geq 4$. The overarching goal will be to use the ingredients from \S2 to derive a contradiction with the Castelnuovo-type bounds in \S1. A fact that we will use at some point is the following: a non-degenerate curve $C \subset \P^{r}$ always satisfies $\deg C \geq r$, and equality holds iff $C$ is a rational normal curve. Moreover, $\deg C = r+1$ can hold only if $p_{g}(C) \leq 1$. Since $X$ does not contain any rational curves \cite{Voisin96}, from the condition $\dim \Span(C) \geq 3$ we know that $\delta \geq 4$. Let $f \colon C' \rightarrow C$ denote the normalization. Assume for contradiction that $\dim \Span(C) \geq 3$.

\noindent \underline{Step 1}: Let us show that $p_{g}(C) \geq 4$.

First assume that $(n, d) \neq (4, 8)$. Since $\delta \geq 4$, from \cite[Theorem 6.1]{CR22} and the fact that $X$ does not contain any lines it follows that $p_{g}(C) \geq 3$. (It is not strictly necessary to use \cite{CR22}, but doing so simplifies the argument.) We can also rule out $\delta = 4, 5$ because otherwise $\dim \Span(C) \geq 3$ implies (by \Cref{prop:Castelnuovo'sThm} for $r = 3$) that $p_{g}(C) \leq 2$, which contradicts the previous sentence. Thus, we may assume that $\delta \geq 6$, and the same result \cite[Theorem 6.1]{CR22} shows that $p_{g}(C) \geq 4$.

Now let us analyze the case $(n,d) = (4,8)$ separately. We claim that a very general such $X$ does not contain any curve $C$ of degree $\delta \leq d + 2 \leq 10$ and geometric genus $p_{g}(C) \leq 3$. Suppose for contradiction that $X$ contained such a curve. By applying Proposition~\ref{prop:ArithmeticGenusIncidence} with $(n,d) = (4,8)$ and using the fact that $h^{0}(f^{\ast}\O_{C}(1)) \leq \delta$ (because $C$ is not rational), we get
\begin{align*}
p_{a}(C) &\geq d\delta - (r+1)(n+1-r+h^{0}(f^{\ast}\O_{C}(1))) - 4 p_{g}(C) + 5 \\
&\geq (7-r)\delta - (r+1)(5-r) - 4 p_{g}(C) + 5.
\end{align*}
For $p_{g}(C) = 1, 2, 3$ and $3 \leq \dim \Span(C) \leq 5$, one can check that this contradicts the Castelnuovo bounds \eqref{eq:anyr>=3Castelnuovo}, \eqref{eq:r=4Castelnuovo}, \eqref{eq:r=3Castelnuovo} from \S1 when $\delta \leq 10$.

Thus, from now on we may assume that $p_{g}(C) \geq 4$.

\noindent \underline{Step 2}: We claim that $\dim \Span(C) \leq 4$.

To this end, let $r \coloneqq \Span(C)$ and assume for contradiction that $r \geq 5$. We will use the Castelnuovo inequality \eqref{eq:anyr>=3Castelnuovo} to arrive at a contradiction. By projecting from points on $C$ into $\P^{2}$ and comparing the resulting genus of the plane model with the bound from Step 1, we see that $\delta \geq r+3$. Riemann--Roch implies that $h^{0}(f^{\ast}\O_{C}(1)) = \delta + 1 - p_{g}(C) + h^{1}(f^{\ast}\O_{C}(1))$, and if $f^{\ast}\O_{C}(1)$ is special then Clifford's theorem says $h^{0}(f^{\ast}\O_{C}(1)) \leq \lfloor \frac{\delta}{2} \rfloor + 1$. We will break up the rest of this argument into two cases.

\underline{Case 2(a)}: Suppose $f^{\ast}\O_{C}(1)$ is nonspecial, i.e. $h^{1}(f^{\ast}\O_{C}(1)) = 0$. Then \eqref{eq:r=n+1arithmeticgenusbound} together with Riemann--Roch gives
\begin{align*}
p_{a}(C) &\geq d \delta - (n+2)(\delta + 1 - p_{g}(C)) - 4 p_{g}(C) + 5 \\
&= (d-n-2) \delta + (n-2)(p_{g}(C)-1) + 1 \\
&\geq (d-n-2) \delta + 3(r-3) + 1 \\
&\geq \left(\frac{1}{2}(\delta-2) - 2 \right) \delta + 3(r-3) + 1.
\end{align*}
The second to last inequality uses the fact that $n \geq r-1$ and $p_{g}(C) \geq 4$, while the last line uses the fact that $d \geq 2n$ and $\delta \leq d+2$. Comparing this to \eqref{eq:anyr>=3Castelnuovo} gives an inequality of the form $Q(\delta) \geq 0$, where $Q$ is a quadratic expression in $\delta$. For $r \geq 5$, one can show that the roots of $Q$ are imaginary and the leading coefficient of $Q$ is negative, which is a contradiction.

\underline{Case 2(b)}: Suppose $f^{\ast}\O_{C}(1)$ is special. Then \eqref{eq:r=n+1arithmeticgenusbound} together with Clifford's theorem gives
\begin{align*}
p_{a}(C) \geq d \delta - (n+2) \left(\frac{\delta}{2} + 1 \right) - 4 p_{g}(C) + 5 &\implies 5 p_{a}(C) \geq \left( d- \frac{n+2}{2} \right) \delta - n + 3.
\end{align*}
Solving for $p_{a}(C)$, we have
\begin{align*}
    p_{a}(C) & \geq \frac{1}{5} \left(d-\frac{n}{2} - 1 \right) \delta - \frac{n}{5} + \frac{3}{5} \\
    & \geq \frac{1}{5} \left(\frac{3d}{4} - 1 \right) \delta - \frac{d}{10} + \frac{3}{5} \\
    &= \frac{3}{20} d \left(\delta - \frac{2}{3} \right) - \frac{1}{5}\delta + \frac{3}{5}\\
    &\geq \frac{3}{20} (\delta-2) \left(\delta - \frac{2}{3}\right) - \frac{1}{5}\delta + \frac{3}{5}\\
    &=\frac{3}{20}\delta^2-\frac{3}{5}\delta+\frac{4}{5}.
\end{align*}  
Comparing this with \eqref{eq:anyr>=3Castelnuovo} gives an inequality of the form $A \delta^{2} + B \delta + C \geq 0$ which needs to be analyzed a bit more carefully. Here, the coefficients are
\[ A = \frac{1}{2(r-1)} - \frac{3}{20}, \quad B = -\frac{(r+1)}{2(r-1)} + \frac{3}{5}, \quad C = \frac{(r+1)^2}{8(r-1)} - \frac{4}{5}. \]
For $r \geq 5$, the quadratic expression $A \delta^{2} + B \delta + C$ gives a downwards facing parabola, and one can check that the largest root is less than $r+2$. This contradicts our assumption from the beginning of Step 2 that $\delta \geq r+3$, and concludes this step.

\noindent \underline{Step 3}: We claim $\dim \Span(C) \not= 4$.

Suppose for contradiction that $\dim \Span(C) = 4$. Since $p_{g}(C) \geq 4$, from the Castelnuovo inequality \eqref{eq:r=4Castelnuovo} we know that $\delta \geq 8$.

\underline{Case 3(a)}: Suppose $\delta \geq 8$, $p_{g}(C) \geq 4$, and $h^{1}(\O_{C}(1)) = 0$. Then applying Proposition~\ref{prop:ArithmeticGenusIncidence} with $r = 4$ and using Riemann--Roch along the lines of Case 2(a) gives
\begin{align*}
p_{a}(C) &\geq d\delta - 5 \left(n-3 + \delta + 1 - p_{g}(C) + h^{1}(f^*\O_{C}(1)) \right) - 4 p_{g}(C) + 5 \\
&= (d-5)\delta - 5n + p_{g}(C) - 5 h^{1}(f^*\O_{C}(1)) + 15 \\
&\geq (d-5)\delta - 5n + 19 \\
&\geq (d-5) \left(\delta-\frac{5}{2} \right) - \frac{25}{2} + 19\\
&\geq (\delta-7) \left(\delta-\frac{5}{2} \right) - \frac{25}{2} + 19\\
&=\delta^2-\frac{19}{2}\delta+24,
\end{align*}
which is a contradiction of \eqref{eq:r=4Castelnuovo}.

\underline{Case 3(b)}: Suppose $\delta \geq 8$, $p_{g}(C) \geq 4$, and $h^{1}(\O_{C}(1)) \not= 0$. Then applying Proposition~\ref{prop:ArithmeticGenusIncidence} with $r = 4$ and using Clifford's theorem in a similar way to Case 2(b) gives
\begin{align*}
p_{a}(C) &\geq \frac{1}{5} \left(d - \frac{r+1}{2} \right) \delta - \frac{1}{5} (r+1)(n+2-r) + 1 \\
&\geq \frac{1}{5} \left(d - \frac{5}{2} \right) \delta - (n-2)+1 \\
&\geq \left(\frac{1}{5}d - \frac{1}{2} \right) \left(\delta - \frac{5}{2} \right) + \frac{7}{4}\\
&\geq \left( \frac{1}{5}(\delta-2) - \frac{1}{2} \right) \left( \delta - \frac{5}{2} \right) + \frac{7}{4}\\
&=\frac{1}{5}\delta^2-\frac{7}{5}\delta+4,
\end{align*}
which contradicts \eqref{eq:r=4Castelnuovo}.

\noindent \underline{Step 4:} Suppose $\dim \Span(C) = 3$. Then we claim that $C$ is not contained in a surface $S \subset \P^{3}$ of degree at most $3$. Moreover, if $\delta \geq 10$ then $C$ is not contained in a surface of degree $4$. Proposition~\ref{prop:incidenceQuadricSurf} shows that $C$ is not contained in a quadric surface. Assume for contradiction that $C$ is contained in a cubic surface. The dimension of the space of cubic surfaces in $\P^{3}$ is $19$. Following the set-up in Proposition~\ref{prop:arithmeticGenusIncidenceSurf}, we see that
\[ p_{a}(C) \geq d \delta - 4(n-2) - 19 - \frac{\delta^{2}}{3} \geq \frac{2}{3}\delta^{2} - 4\delta - 7, \]
where the last inequality follows from the fact that $\delta \leq d+2$ and $d \geq 2n$. This contradicts Halphen's bound \eqref{eq:HalphenCubicSurf} when $\delta \geq 9$. Since $d \geq 2n \geq 8$, $\delta \leq 8$ implies that $\delta \leq d$. This leads to an improved bound of $p_{a}(C) \geq \frac{2}{3}\delta^{2} - 2\delta - 11$, which allows us to reach a contradiction when $\delta \geq 7$. Thus, the remaining cases to work out are when $\delta = 5, 6$ (note that if $\delta = 4$, then $C$ is contained in a quadric). Both of these remaining two cases can be handled by noting that $p_{g}(C) \geq 5$ -- indeed, if $p_{g}(C) = 4$, then applying Proposition~\ref{prop:ArithmeticGenusIncidence} with $r = 3$ and using either Riemann--Roch or Clifford's theorem to bound $h^{0}(f^{\ast}\O_{C}(1))$ as before leads to an impossible lower bound on $p_{a}(C)$. But $p_{g}(C) \geq 5$ contradicts \eqref{eq:HalphenCubicSurf}.

Next, let us suppose for contradiction that $\delta \geq 10$ and $C$ is contained in a quartic surface. The dimension of the space of quartic surfaces in $\P^{3}$ is $34$. Again following the set-up in Proposition~\ref{prop:arithmeticGenusIncidenceSurf}, we see that
\[ p_{a}(C) \geq d \delta - 4(n-2) - 34 - \frac{\delta^{2}}{4} \geq \frac{3}{4} \delta^{2} - 4 \delta - 22, \]
where the last inequality follows from the fact that $\delta \leq d+2$ and $d \geq 2n$. This contradicts the bound in \eqref{eq:r=3Harris}, which can be simplified to
\[ p_{a}(C) \leq \begin{cases} \frac{\delta^{2}}{8} + 1 & \text{if } \delta > 12 \\ \frac{\delta^{2}}{6} - \frac{\delta}{2} + 1 & \text{if } 10 \leq \delta \leq 12. \end{cases} \]

\noindent \underline{Step 5}: Still assuming $\dim \Span(C) = 3$, we will derive a contradiction with Step 4.

We will first show that $\delta \geq 8$. Note that $\delta \geq 7$, because otherwise $H^{0}(\O_{\P^{3}}(3)) = 20$ so we would be able to find a cubic surface $S$ meeting $C$ in at least 19 points. By B\'{e}zout's theorem, $C$ is contained in $S$, which is a contradiction. Next, we will need a result of Gruson-Peskine \cite[Th\'{e}or\`{e}me de Sp\'{e}cialit\'{e}]{GP78}, which says that if $C \subset \P^{3}$ is an integral curve of degree $\delta$ which is not contained in a surface of degree less than $s$, then for any integer $n \geq s + \frac{\delta}{s} - 4$, the line bundle $\O_{C}(n)$ is special if and only if $n = s + \frac{\delta}{s} - 4$ and $C$ is a complete intersection curve of type $(s, \frac{\delta}{s})$. Consider the two situations below:
\begin{enumerate}[(i)]
    \item Suppose $\delta \leq 12$. Apply the theorem above with $n = 3$, $s = 3$. Since $C$ is not contained in surface of degree at most $3$, in particular it cannot be a complete intersection curve of type $(s, \frac{\delta}{s})$ and so $\O_{C}(3)$ is nonspecial.
    \item Suppose $\delta \leq 15$. Apply the theorem above with $n = 4$, $s = 3$. Since $C$ is not a complete intersection curve of type $(s, \frac{\delta}{s})$, as before we see that $\O_{C}(4)$ is nonspecial (note that for $\delta = 11, 12$ this follows from the previous bullet point).
\end{enumerate}
Both of these observations will be used later. For now, we note that Riemann--Roch and (i) imply that if $\delta = 7$, then
\[ h^{0}(\O_{C}(3)) = 3 \delta + 1 - p_{a}(C) = 22 - p_{a}(C) < 20 = h^{0}(\O_{\P^{3}}(3)). \]
The middle inequality follows from Step 1, which showed that $p_{a}(C) \geq p_{g}(C) \geq 4$. But this contradicts the fact that $C$ is not contained in a cubic surface. So $\delta \geq 8$ as claimed.

\underline{Case 5(a)}: Suppose $\delta \geq 8$, $p_{g}(C) \geq 4$, and $h^{1}(f^{\ast}\O_{C}(1)) = 0$. By applying Proposition~\ref{prop:ArithmeticGenusIncidence} with $r = 3$ and using Riemann--Roch (along the lines of Case 2(a)) gives
\begin{align*}
p_{a}(C) &\geq d\delta - 4 (n-2 + \delta + 1 - p_{g}(C) ) - 4 p_{g}(C) + 5 \\
&= (d-4)\delta - 4n  + 9 \\
&\geq (\delta-6)(\delta-2) + 1
\end{align*}
where the last line uses the fact that $d \geq 2n$. This contradicts \eqref{eq:GrusonPeskineNotInQuadricCubicSurf} (assuming $C$ is not contained in a surface $S \subset \P^{3}$ of degree at most $3$) when $\delta \geq 8$.

\underline{Case 5(b)}: Suppose $\delta \geq 8$, $p_{g}(C) \geq 4$, and $h^{1}(f^{\ast}\O_{C}(1)) \not= 0$. Then applying Proposition~\ref{prop:ArithmeticGenusIncidence} with $r = 3$ and using Clifford's theorem to bound $h^{0}(f^{\ast}\O_{C}(1)) \leq \lfloor \frac{\delta}{2} \rfloor + 1$ (similar to Case 2(b)) gives
\begin{align*}
p_{a}(C) &\geq \frac{1}{5} \left[ d \delta - 4 \left( n-2 + \left\lfloor \frac{\delta}{2} \right\rfloor + 1 \right) + 5 \right] \\
&\geq \frac{1}{5} \left[ d (\delta-2) - 4 \left\lfloor \frac{\delta}{2} \right\rfloor + 9 \right].
\end{align*}
Plugging in $d\geq\delta-2$ and comparing this with \eqref{eq:GrusonPeskineNotInQuadricCubicSurf}, we arrive at a contradiction when $\delta \geq 15$. We will finish the cases where $\dim \Span(C) = 3$ and $8 \leq \delta \leq 14$ by hand.

When $8 \leq \delta \leq 9$ and $\delta \leq d+2$, one can apply (i) plus the arithmetic genus bound above to Riemann--Roch, which gives $h^{0}(\O_{C}(3))  = 3\delta + 1 - p_{a}(C) < 20 = h^{0}(\O_{\P^{3}}(3))$. This contradicts the fact that $C$ is not contained in a cubic surface. When $10 \leq \delta \leq 14$, the same strategy using (ii) implies that $h^{0}(\O_{C}(4)) < 35 = h^{0}(\O_{\P^{3}}(4))$. Therefore, $C$ is contained in a quartic surface, which contradicts Step 4. This completes the proof. \qed

\section{Subvarieties of higher degree}

In this section, we would like to prove a number of results which extend beyond curves of degree at most $d+2$. We begin with the following:

\begin{theorem}\label{thm:curvesHigherDegree}
    Fix any $s \in \Z_{\geq 1}$. Then there exists a positive integer $d_{0} = d_{0}(s, n)$ such that if $d \geq d_{0}$ and $X \subset \P^{n+1}$ is a general hypersurface of degree $d$ and dimension $n \geq 3$, then any integral curve $C \subset X$ of degree $\delta \leq sd$ is equal to a generically transverse intersection $X \cap S$ for a unique surface $S$ of degree at most $s$. Furthermore, one can choose $d_{0}(1, n) = 2n$ and
    \[ d_{0}(s,n) \coloneqq \max \left\{(s+1)(s + 3n -1), \ \frac{2(s+1)}{(n-1)(s-1)} \prod_{i=1}^{n-1} \sqrt[n-i]{n! (s+1)} \right\} \quad \text{ for } s \geq 2. \]
\end{theorem}

\begin{proof}
    We will prove this by induction on $s$. The base case $s = 1$ follows from Theorem~\ref{thm:planeCurvesMinimal}. Let us assume that the theorem holds for all $s \leq k-1$ (for some positive integer $k$). We would like to prove the theorem for $s = k$. Let $C \subset X$ be a curve of degree $\delta \leq kd$ contained in a very general hypersurface of degree $d$. By the inductive hypothesis, we may assume that $(k-1)d + 1 \leq \delta \leq kd$.

    \textbf{Claim 1.} If $d \geq d_{0}(k,n)$, then $C$ is contained in a surface $S \subset \P^{n+1}$ of degree at most $k$.

    To see this, suppose for contradiction that $C$ is not contained in such a surface. The assumption on $d$ and $\delta \geq (k-1)d + 1$ implies that
    \[ \delta > \frac{2(k+1)}{n-1} \prod_{i=1}^{n-1} \sqrt[n-i]{n! (k+1)}. \]
    Given this lower bound on $\delta$, the main result of \cite{CCDG93} is the upper bound
    \begin{equation}\label{eq:genusboundprelim}
    p_a(C)\leq 1+\frac{\delta}{2}(w+m-2)-\frac{m+1}{2}(w-\epsilon)+(w+1)\frac{vm}{2}+\rho,
    \end{equation}
    where $m,\epsilon, w, v,$ and $\rho$ are defined in \cite{CCDG93}. (Note that our conventions differ from that of \cite{CCDG93}; since our assumption involves at most $k$, in their definitions of $m,\epsilon,w$, and $v$ we must replace $s$ with $k+1$ everywhere.)

    We will weaken (\ref{eq:genusboundprelim}) to a more convenient bound which does not require the definition of any new variables. Using the naive inequalities
    \[ m \leq \frac{\delta-1}{k+1}, \quad w \leq \frac{k}{n-1}\leq k, \quad v \leq n-2, \quad \epsilon \leq k, \quad \rho \leq \frac{\epsilon}{2}\leq\frac{k}{2}, \quad \text{and} \ w-\epsilon\geq-\epsilon\geq-k, \]
    the bound \eqref{eq:genusboundprelim} implies that
    \begin{align}\label{eq:genusboundinduction}
    p_{a}(C) &\leq 1+\frac{\delta}{2}(w+m-2)+\frac{m+1}{2}\cdot k+(w+1)\frac{vm}{2}+\rho \nonumber \\
    &\leq1+\frac{\delta}{2} \left(k+\frac{\delta-1}{k+1}-2 \right)+\frac{k(\delta+k)}{2(k+1)}+\frac{(n-2)(\delta-1)}{2}+\frac{k}{2}\nonumber\\
    &=\frac{\delta^2}{2(k+1)}+ \left(\frac{k}{2}+\frac{k-1}{2(k+1)}-1+\frac{n-2}{2} \right) \delta+\frac{k^2}{2(k+1)}-\frac{(n-2)}{2}+\frac{k}{2}+1\nonumber\\
    &\leq\frac{\delta^2}{2(k+1)}+\frac{k+n-3}{2}\cdot\delta+\frac{2k-n+4}{2}.
    \end{align}
    
    Strictly speaking, the main result of \cite{CCDG93} assumes that $C$ is a non-degenerate curve in $\P^{n+1}$. However, if $C$ spans a lower-dimensional linear subspace $\Lambda$, we claim that the bound still applies. The reader can verify that the assumption on $\delta$ is even easier to satisfy (the lower bound on $\delta$ decreases as $n$ decreases), $C \subset \Lambda$ is still not contained in a surface of degree at most $k$, and the bound on the arithmetic genus in \eqref{eq:genusboundinduction} becomes stronger as $n$ decreases. For this last part, observe that the right-hand side is a linear function in $n$ with slope $\frac{1}{2} (\delta - 1)\geq 0$.
    
    On the other hand, Voisin \cite{Voisin96} (cf. \cite[Proposition 3.1]{Chen24}) has shown that on a very general hypersurface $X \subset \P^{n+1}$ of degree $d \geq 2n$, the geometric genus of any integral curve $C \subset X$ is bounded from below by
    \[ p_{g}(C) \geq 1 + \frac{1}{2} (d-2n-1) \delta. \]
    We can now use the inequality $p_{g}(C) \leq p_{a}(C)$ to see that
    \begin{align*}
    1 + \frac{1}{2} (d-2n-1)\delta&\leq \frac{\delta^2}{2(k+1)}+\frac{k+n-3}{2}\cdot\delta+\frac{2k-n+4}{2},
    \end{align*}
    so
    \begin{align*}
    0&\leq\frac{\delta^2}{2(k+1)}+\frac{k+3n-2-d}{2}\cdot\delta+\frac{2k-n+2}{2}\\
    0&\leq\frac{\delta}{k+1}+k+3n-2-d+\frac{2k-n+2}{\delta}.
    \end{align*}
    We have the simple bound
    \[
    \frac{2k-n+2}{\delta}<\frac{2k}{\delta}<\frac{2k}{(k-1)d}\leq\frac{4}{d}<1,
    \]
    so we get
    \[
    d<\frac{\delta}{k+1}+k+3n-1,
    \]
    or
    \[
    (k+1)d-\delta<(k+1)(k+3n-1).
    \]
    Because $\delta \leq kd$, we have $(k+1)d-\delta\geq d$ so $d < (k+1)(k + 3n-1)$, which contradicts our assumption $d \geq d_{0}(k, n)$.

    From Theorem~\ref{thm:main}, we see that $X \subset \P^{n+1}$ does not contain any surfaces of degree at most $k$ since $d \geq 2n$ and $d \geq k+1$. Since $C$ is contained in a surface of degree at most $k$, it suffices to show the following.
    
    \textbf{Claim 2.} If $d \geq d_{0}(k,n)$, then the intersection of $X$ with \textit{any} integral surface $S \subset \P^{n+1}$ of degree at most $k$ is integral.

    Suppose for contradiction that there exists an integral surface $S \subset \P^{n+1}$ of degree at most $k$ such that $S \cap X$ is not integral. Then we may write the $1$-cycle $S \cap X$ as a nontrivial sum $\sum_{i=1}^{m} C_{i}$ of a finite number of irreducible curves. Note that some of the $C_{i}$ may be the same if some irreducible component of the intersection is nonreduced. The condition $m \geq 2$ and the fact that $d \geq 2n$ guarantee that for each $i$, we have $d \leq \deg C_{i} \leq (k-1)d$. Here, the lower bound on $\deg C_{i}$ follows from Theorem~\ref{thm:planeCurvesMinimal}. The inductive hypothesis then implies that each component $C_{i}$ is equal to a generically transverse intersection $T_{i} \cap X$, for some surfaces $T_{i}$ of degree $a_{i}$ with $\sum a_{i} = \deg S$. By degree reasons, $T_{i} \not= S$ for all $i$. But then for any $i$ we have $C_{i} \subset S \cap  T_{i}$,
    which implies that
    \[ d\cdot a_i\leq\deg S \cap  T_{i} \leq \deg S \cdot a_{i} \leq k\cdot a_i < d\cdot a_i, \]
    a contradiction. Uniqueness follows from the same argument as in the proof of Theorem~\ref{thm:B}.
\end{proof}

\begin{remark}\label{rem:constantd_0}
The last term in the expression for $d_{0}(s,n)$ in Theorem~\ref{thm:curvesHigherDegree} grows approximately like $C \cdot s^{1 + \frac{1}{2} + \cdots + \frac{1}{n-1}}$, where $C = C(n)$ is a constant depending only on $n$. For $n \leq 4$, this implies that $d_{0}(s,n)$ grows quadratically in $s$. The proof below will show that the same constants that appear above for $d_{0}(s,n)$ can be used in the statement of Theorem~\ref{thm:B}.
\end{remark}

Next, we will give the proof of Theorem~\ref{thm:B}.

\begin{proof}[Proof of Theorem~\ref{thm:B}]
    Choose $d_{0}$ as in Theorem~\ref{thm:curvesHigherDegree} or Remark~\ref{rem:constantd_0}. First, we note that the uniqueness of $V$ is automatic. Indeed, if there were two subvarieties $V$ and $V'$ of degrees $\frac{\delta}{d}$ with $Y=X\cap V=X\cap V',$ then $Y$ would be a connected component of $V\cap V'.$ But then $\deg V \cdot \deg V'=\frac{\delta^2}{d^2}\geq\deg Y=\delta$, so we would have $\delta\geq d^2$ and $s\geq d$, which is impossible for our choice of $d_0$.

    For the rest of the proof, we will run an induction argument on $\dim Y$ and $\dim X$ simultaneously. The base case where $\dim Y = 1$ follows from Theorem~\ref{thm:curvesHigherDegree}, and the case where $\dim Y = n - 1$ follows from the Lefschetz hyperplane theorem. Thus, the first nontrivial case is when $n \geq 4$.
    
    Let $Y \subset X$ be a subvariety of degree $\delta \leq ds$ and dimension $m$, where $2 \leq m \leq n-2$. Fix a general pencil $\{ H_{t} \}_{t \in \P^{1}}$ of hyperplanes in $\P^{n+1}$, and let $\Delta = \bigcap_{t \in \P^{1}} H_{t}$ be its base locus. We may choose the pencil so that for all $t \in \P^{1}$,
    \begin{enumerate}[(i)]
        \item $Y_{t} \coloneqq Y \cap H_{t}$ has pure dimension $m-1$, and
        \item no irreducible component of $Y_{t}$ is contained in $\Delta$.
    \end{enumerate}
    Applying the inductive hypothesis to $Y_{t} \subset X$ gives a unique subvariety $V_{t} \subset \P^{n+1}$ of dimension $m$ and degree $e \leq s$ with the property that $Y \cap H_{t} = X \cap V_{t}$. In the case where $Y_{t}$ is reducible, we can apply the inductive hypothesis to each irreducible component of $Y_{t}$ and take the resulting union to form $V_{t}$. For a general $t \in \P^{1}$, by the uniqueness of $V_{t}$ this is equivalent to applying the inductive hypothesis to $Y_{t} \subset X \cap H_{t}$, where $X \cap H_{t}$ is a general hypersurface of dimension one less. Thus, we see that $\delta = de$ and $V_{t} \subset H_{t}$ for a general $t \in \P^{1}$.
    
    Next, there exists a dense open $U \subset \P^{1}$ such that the $V_{t}$ all belong to the same irreducible component of the Hilbert scheme $\H$, which gives a map $U \rightarrow \H$. By projectivity of Hilbert schemes, this extends to a map $\P^{1} \rightarrow \H$, and we can pull back the universal family over $\H$ to get a flat family $W \subset \P^{n+1} \times \P^{1}$ that fits into the diagram below.
    \begin{center}
    \begin{tikzcd}
    & W \arrow[ld, swap, "h"] \arrow[rd] \arrow[r, hook] & \Bl_{\Delta} \P^{n+1} \arrow[r, phantom, "\subset"] \arrow[d] &[-0.2in] \P^{n+1} \times \P^{1} \arrow[ld] \\
    \P^{n+1} & & \P^{1}
    \end{tikzcd}
    \end{center}
    For $t \in U$, the fiber $W_{t}$ maps isomorphically onto $V_{t}$ by construction. By flatness, this implies that for \textit{every} $t \in \P^{1}$, $W_{t}$ has pure dimension $m$ and degree $e$. The morphism $W \rightarrow \P^{1}$ is open since it is a flat morphism of finite type. By \cite[\href{https://stacks.math.columbia.edu/tag/004Z}{Tag 004Z}]{stacks-project}, $W$ is irreducible and thus $V \coloneqq h(W)$ is also irreducible.
    
    Let $V \coloneqq h(W)$. The rest of this proof is dedicated to showing that $V \cap X = Y$. As an intermediate step, we will need the following
    
    \textbf{Claim.} $V_{t} = h(W_{t})$ and $V_{t} = V \cap H_{t}$ for all $t \in \P^{1}$.
    
    Note that $Y_{t} \subset h(W_{t})$ for all $t \in \P^{1}$ since it is true for a general $t \in \P^{1}$ and $V \coloneqq h(W)$ contains $Y$. Since $e \leq s < d$, \Cref{thm:main} implies that for every $t \in \P^{1}$, no irreducible component of $h(W_t)$ is contained in $X$. Therefore, $h(W_{t}) \cap X$ has pure dimension $m - 1$ and degree $de$ (counted with multiplicity) and contains $Y_{t}$, which has degree $de$ (counted with multiplicity). For degree reasons, it follows that $h(W_{t}) \cap X = Y_{t}$. By construction, $V_{t}$ has degree $de$ for all $t \in \P^{1}$ and satisfies the same property $V_{t} \cap X = Y_{t}$. By a uniqueness argument similar to the first paragraph, it follows that $V_{t} = h(W_{t})$ for all $t \in \P^{1}$. This proves the first part of the claim. Moreover, by construction the irreducible components of $V_{t}$ are in bijection with the irreducible components of $Y_{t}$. Property (ii) tells us that no component of $Y_{t}$ is contained in $\Delta$, so no component of $V_{t}$ is contained in $\Delta$. Since any extra components of $V \cap H_{t}$ beyond $V_{t}$ would have to be contained in another $H_{t'}$, and hence in $\Delta$, this implies that $V_{t} = V \cap H_{t}$ for all $t \in \P^{1}$. The proof of the Claim is complete.
    
    Now we will show that $X \cap V = Y$. Assume for contradiction that $X \cap V = Y \cup G$, where $G$ is a variety of dimension $m$. For simplicity let us assume that $G$ is irreducible. The map $\Bl_{\Delta} \P^{n+1} \rightarrow \P^{n+1}$ is an isomorphism outside of the exceptional divisor, which maps down to $\Delta$ as a $\P^{1}$-bundle. Using the fact that $W$ is irreducible and not contained in the exceptional divisor over $\Delta$, we observe that the strict transform $\overline{G} \subset W$ of $G$ is a subvariety of dimension $m$ and so $G$ cannot be contained in $\Delta$. For any fixed $t \in \P^{1}$, from the previous paragraph we know that no irreducible component of $h(W_t)$ is contained in $X$, so $\overline{G}$ is not contained in $W_{t}$. Therefore, $\overline{G}$ must dominate $\P^{1}$, which implies that $G$ meets $H_{t'}$ properly for a general $t' \in \P^{1}$. By the Claim and our assumption above, we can write
    \[ Y_{t'} = X \cap V_{t'} = X \cap V \cap H_{t'} = (Y \cup G) \cap H_{t'} = (Y \cap H_{t'}) \cup (G \cap H_{t'}) = Y_{t'} \cup (G \cap H_{t'}), \]
    which is a contradiction since $G \cap H_{t'}$ is not contained in $Y_{t'}$. When $G$ is not irreducible, the same argument can be applied to some irreducible component of $G$. Thus, we have $X \cap V = Y$, which concludes the proof.
\end{proof}

Finally, we believe that the following extension of Theorem~\ref{thm:planeCurvesMinimal} should hold:

\begin{conjecture}
Let $X \subset \P^{n+1}$ be a general hypersurface of degree $d$ and dimension $n \geq 3$ such that $\frac{3}{2}n + 2 \leq d \leq 2n-1$. Then any curve $C \subset X$ of degree $\delta \leq d$ must be contained in a 2-plane.
\end{conjecture}

\noindent By a variation of the argument in Proposition~\ref{prop:incidence2plane} together with the main result in \cite{RY20}, this would imply that $C$ must be one of the following curves: a line, an irreducible plane curve of degree $d-1$ (that is residual to a line), or an irreducible plane curve of degree $d$.

\bibliographystyle{alpha} 
\bibliography{Biblio}

@article {Chen24,
    AUTHOR = {Chen, N.},
     TITLE = {Multiplicative bounds for measures of irrationality on
              complete intersections},
   JOURNAL = {Algebr. Geom.},
  FJOURNAL = {Algebraic Geometry},
    VOLUME = {11},
      YEAR = {2024},
    NUMBER = {5},
     PAGES = {737--756},
      ISSN = {2313-1691,2214-2584},
   MRCLASS = {14M10},
  MRNUMBER = {4791072},
}

@article{CCZ24,
  title={Curves on complete intersections and measures of irrationality},
  author={Chen, N. and Church, B. and Zhao, J.},
  journal={arXiv preprint arXiv:2406.12101},
  year={2024}
}

@article {CCDG93,
    AUTHOR = {Chiantini, L. and Ciliberto, C. and Di Gennaro, V.},
     TITLE = {The genus of projective curves},
   JOURNAL = {Duke Math. J.},
  FJOURNAL = {Duke Mathematical Journal},
    VOLUME = {70},
      YEAR = {1993},
    NUMBER = {2},
     PAGES = {229--245},
      ISSN = {0012-7094,1547-7398},
   MRCLASS = {14H99},
  MRNUMBER = {1219813},
MRREVIEWER = {Giorgio\ Bolondi},
       DOI = {10.1215/S0012-7094-93-07003-2},
       URL = {https://doi.org/10.1215/S0012-7094-93-07003-2},
}

@article {Clemens86,
    AUTHOR = {Clemens, H.},
     TITLE = {Curves on generic hypersurfaces},
   JOURNAL = {Ann. Sci. \'Ecole Norm. Sup. (4)},
  FJOURNAL = {Annales Scientifiques de l'\'Ecole Normale Sup\'erieure.
              Quatri\`eme S\'erie},
    VOLUME = {19},
      YEAR = {1986},
    NUMBER = {4},
     PAGES = {629--636},
      ISSN = {0012-9593},
   MRCLASS = {14H10 (14F05)},
  MRNUMBER = {875091},
MRREVIEWER = {Werner\ Kleinert},
       URL = {http://www.numdam.org/item?id=ASENS_1986_4_19_4_629_0},
}

@article {CR04,
    AUTHOR = {Clemens, H. and Ran, Z.},
     TITLE = {Twisted genus bounds for subvarieties of generic
              hypersurfaces},
   JOURNAL = {Amer. J. Math.},
  FJOURNAL = {American Journal of Mathematics},
    VOLUME = {126},
      YEAR = {2004},
    NUMBER = {1},
     PAGES = {89--120},
      ISSN = {0002-9327,1080-6377},
   MRCLASS = {14J70 (14E15)},
  MRNUMBER = {2033565},
MRREVIEWER = {Angelo\ Lopez},
       URL =
              {http://muse.jhu.edu/journals/american_journal_of_mathematics/v126/126.1clemens.pdf},
}

@article {CR19,
    AUTHOR = {Coskun, I. and Riedl, E.},
     TITLE = {Algebraic hyperbolicity of the very general quintic surface in
              {$\Bbb{P}^3$}},
   JOURNAL = {Adv. Math.},
  FJOURNAL = {Advances in Mathematics},
    VOLUME = {350},
      YEAR = {2019},
     PAGES = {1314--1323},
      ISSN = {0001-8708,1090-2082},
   MRCLASS = {32Q45 (14H50 14J29 14J70)},
  MRNUMBER = {3949983},
MRREVIEWER = {Simone\ Diverio},
       DOI = {10.1016/j.aim.2019.04.062},
       URL = {https://doi.org/10.1016/j.aim.2019.04.062},
}

@article {CR22,
    AUTHOR = {Coskun, I. and Riedl, E.},
     TITLE = {Clustered families and applications to {L}ang-type
              conjectures},
   JOURNAL = {Proc. Lond. Math. Soc. (3)},
  FJOURNAL = {Proceedings of the London Mathematical Society. Third Series},
    VOLUME = {125},
      YEAR = {2022},
    NUMBER = {6},
     PAGES = {1353--1376},
      ISSN = {0024-6115,1460-244X},
   MRCLASS = {14J70 (14G05 14M15 32Q45)},
  MRNUMBER = {4524219},
MRREVIEWER = {Fyodor\ L.\ Zak},
       DOI = {10.1112/plms.12484},
       URL = {https://doi.org/10.1112/plms.12484},
}

@article {Ein88,
    AUTHOR = {Ein, L.},
     TITLE = {Subvarieties of generic complete intersections},
   JOURNAL = {Invent. Math.},
  FJOURNAL = {Inventiones Mathematicae},
    VOLUME = {94},
      YEAR = {1988},
    NUMBER = {1},
     PAGES = {163--169},
      ISSN = {0020-9910,1432-1297},
   MRCLASS = {14C05 (14M10)},
  MRNUMBER = {958594},
MRREVIEWER = {Paltin\ Ionescu},
       DOI = {10.1007/BF01394349},
       URL = {https://doi.org/10.1007/BF01394349},
}

@article {GH85,
    AUTHOR = {Griffiths, P. and Harris, J.},
     TITLE = {On the {N}oether-{L}efschetz theorem and some remarks on
              codimension-two cycles},
   JOURNAL = {Math. Ann.},
  FJOURNAL = {Mathematische Annalen},
    VOLUME = {271},
      YEAR = {1985},
    NUMBER = {1},
     PAGES = {31--51},
      ISSN = {0025-5831,1432-1807},
   MRCLASS = {14J05 (14C30)},
  MRNUMBER = {779603},
MRREVIEWER = {I.\ Dolgachev},
       DOI = {10.1007/BF01455794},
       URL = {https://doi.org/10.1007/BF01455794},
}

@article {GLP83,
    AUTHOR = {Gruson, L. and Lazarsfeld, R. and Peskine, C.},
     TITLE = {On a theorem of {C}astelnuovo, and the equations defining
              space curves},
   JOURNAL = {Invent. Math.},
  FJOURNAL = {Inventiones Mathematicae},
    VOLUME = {72},
      YEAR = {1983},
    NUMBER = {3},
     PAGES = {491--506},
      ISSN = {0020-9910,1432-1297},
   MRCLASS = {14H45 (14F05)},
  MRNUMBER = {704401},
MRREVIEWER = {Fabrizio\ Catanese},
       DOI = {10.1007/BF01398398},
       URL = {https://doi.org/10.1007/BF01398398},
}

@incollection {GP78,
    AUTHOR = {Gruson, L. and Peskine, C.},
     TITLE = {Genre des courbes de l'espace projectif},
 BOOKTITLE = {Algebraic geometry ({P}roc. {S}ympos., {U}niv. {T}roms\o,
              {T}roms\o, 1977)},
    SERIES = {Lecture Notes in Math.},
    VOLUME = {687},
     PAGES = {31--59},
 PUBLISHER = {Springer, Berlin},
      YEAR = {1978},
      ISBN = {3-540-08954-3},
   MRCLASS = {14H45},
  MRNUMBER = {527229},
MRREVIEWER = {Paolo\ Maroscia},
}

@article{Halphen21,
  title={Classification des courbes gauches alg{\'e}briques},
  author={Halphen, C. H.},
  journal={Oeuvres III, Gauthier-VUlars, Paris},
  year={1921}
}

@article {Harris80,
    AUTHOR = {Harris, J.},
     TITLE = {The genus of space curves},
   JOURNAL = {Math. Ann.},
  FJOURNAL = {Mathematische Annalen},
    VOLUME = {249},
      YEAR = {1980},
    NUMBER = {3},
     PAGES = {191--204},
      ISSN = {0025-5831,1432-1807},
   MRCLASS = {14H45},
  MRNUMBER = {579101},
MRREVIEWER = {Ciro\ Ciliberto},
       DOI = {10.1007/BF01363895},
       URL = {https://doi.org/10.1007/BF01363895},
}

@book {Lazarsfeld04a,
    AUTHOR = {Lazarsfeld, Robert},
     TITLE = {Positivity in algebraic geometry. {I}},
    VOLUME = {48},
      NOTE = {Classical setting: line bundles and linear series},
 PUBLISHER = {Springer-Verlag, Berlin},
      YEAR = {2004},
     PAGES = {xviii+387},
      ISBN = {3-540-22533-1},
   MRCLASS = {14-02 (14C20)},
  MRNUMBER = {2095471},
MRREVIEWER = {Mihnea\ Popa},
       DOI = {10.1007/978-3-642-18808-4},
       URL = {https://doi.org/10.1007/978-3-642-18808-4},
}

@article {Pacienza04,
    AUTHOR = {Pacienza, G.},
     TITLE = {Subvarieties of general type on a general projective
              hypersurface},
   JOURNAL = {Trans. Amer. Math. Soc.},
  FJOURNAL = {Transactions of the American Mathematical Society},
    VOLUME = {356},
      YEAR = {2004},
    NUMBER = {7},
     PAGES = {2649--2661},
      ISSN = {0002-9947,1088-6850},
   MRCLASS = {14J70 (14K12 32Q45)},
  MRNUMBER = {2052191},
MRREVIEWER = {Massimiliano\ Mella},
       DOI = {10.1090/S0002-9947-03-03250-1},
       URL = {https://doi.org/10.1090/S0002-9947-03-03250-1},
}

@article {RY20,
    AUTHOR = {Riedl, E. and Yang, D.},
     TITLE = {Rational curves on general type hypersurfaces},
   JOURNAL = {J. Differential Geom.},
  FJOURNAL = {Journal of Differential Geometry},
    VOLUME = {116},
      YEAR = {2020},
    NUMBER = {2},
     PAGES = {393--403},
      ISSN = {0022-040X,1945-743X},
   MRCLASS = {14J70 (14M15)},
  MRNUMBER = {4168208},
MRREVIEWER = {Roberto\ Mu\~noz},
       DOI = {10.4310/jdg/1603936816},
       URL = {https://doi.org/10.4310/jdg/1603936816},
}

@misc{stacks-project,
  author       = {The {Stacks project authors}},
  title        = {The Stacks project},
  howpublished = {\url{https://stacks.math.columbia.edu}},
  year         = {2026},
}

@incollection {Voisin89,
    AUTHOR = {Voisin, C.},
     TITLE = {Sur une conjecture de {G}riffiths et {H}arris},
 BOOKTITLE = {Algebraic curves and projective geometry ({T}rento, 1988)},
    SERIES = {Lecture Notes in Math.},
    VOLUME = {1389},
     PAGES = {270--275},
 PUBLISHER = {Springer, Berlin},
      YEAR = {1989},
      ISBN = {3-540-51509-7},
   MRCLASS = {14J30 (14C15 14C20 14J70)},
  MRNUMBER = {1023405},
MRREVIEWER = {John\ B.\ Little},
       DOI = {10.1007/BFb0085939},
       URL = {https://doi.org/10.1007/BFb0085939},
}

@article{Voisin96,
  title={On a conjecture of {C}lemens on rational curves on hypersurfaces},
  author={Voisin, C.},
  journal={Journal of Differential Geometry},
  volume={44},
  number={1},
  pages={200--213},
  year={1996},
  publisher={Lehigh University}
}

@article{Wu90,
    AUTHOR = {Wu, X.},
     TITLE = {On a conjecture of {G}riffiths-{H}arris generalizing the
              {N}oether-{L}efschetz theorem},
   JOURNAL = {Duke Math. J.},
  FJOURNAL = {Duke Mathematical Journal},
    VOLUME = {60},
      YEAR = {1990},
    NUMBER = {2},
     PAGES = {465--472},
      ISSN = {0012-7094,1547-7398},
   MRCLASS = {14J70 (14D07 14J15 14M10)},
  MRNUMBER = {1047761},
MRREVIEWER = {Claire\ Voisin},
       DOI = {10.1215/S0012-7094-90-06018-1},
       URL = {https://doi.org/10.1215/S0012-7094-90-06018-1},
}

@article {Xu94,
    AUTHOR = {Xu, G.},
     TITLE = {Subvarieties of general hypersurfaces in projective space},
   JOURNAL = {J. Differential Geom.},
  FJOURNAL = {Journal of Differential Geometry},
    VOLUME = {39},
      YEAR = {1994},
    NUMBER = {1},
     PAGES = {139--172},
      ISSN = {0022-040X,1945-743X},
   MRCLASS = {14J70 (14B05 14H45 14M07)},
  MRNUMBER = {1258918},
MRREVIEWER = {Frank\ Loose},
       URL = {http://projecteuclid.org/euclid.jdg/1214454680},
}

\footnotesize{
\textsc{Department of Mathematics, Harvard University, MA 02138} \\
\indent \textit{E-mail address:} \href{mailto:nathanchen@math.columbia.edu}{nathanchen@math.harvard.edu}

\textsc{Department of Mathematics, Massachusetts Institute of Technology, MA 02139} \\
\indent \textit{E-mail address:} \href{mailto:dhy@mit.edu}{dhy@mit.edu}
}

\end{document}